\documentclass[12pt]{amsart}
\usepackage{geometry}
\usepackage{amsmath}
\usepackage{amssymb}
\usepackage{latexsym}
\usepackage{mathrsfs}
\usepackage{enumitem}

\allowdisplaybreaks
\theoremstyle{lemma}
\newtheorem{theorem}{Theorem}[section]
\newtheorem{lemma}[theorem]{Lemma}

\theoremstyle{definition}
\newtheorem{definition}[theorem]{Definition}

\newtheorem{proposition}[theorem]{Proposition}
\newtheorem{corollary}[theorem]{Corollary}

\theoremstyle{remark}

\theoremstyle{notation}

\theoremstyle{claim}

\numberwithin{equation}{section}

\begin{document}

\title[Invariant Ergodic measures ]{Invariant ergodic  measures and the classification of crossed product $C^\ast$-algebras }

\author{Xin Ma}
\email{dongodel@math.tamu.edu}
\address{Department of Mathematics,
         Texas A\&M University,
         College Station, TX 77843}


\date{June 8, 2018.}




\begin{abstract}
Let $\alpha: G\curvearrowright X$ be a minimal free continuous action of an infinite countable amenable group on an infinite compact metrizable space. In this paper, under the hypothesis that the invariant ergodic probability Borel measure space $E_G(X)$ is compact and zero-dimensional, we show that the action $\alpha$ has the small boundary property. This partially answers an open problem in dynamical systems that asks whether a minimal free action of an amenable group has the small boundary property if its space $M_G(X)$ of invariant Borel probability measures forms a Bauer simplex. In addition, under the same hypothesis, we show that dynamical comparison implies almost finiteness, which was shown by Kerr to imply that the crossed product is $\mathcal{Z}$-stable.  Finally, we discuss some rank properties and provide two classifiability results for crossed products, one of which is based on the work of Elliott and Niu.

\end{abstract}

\maketitle

\section{Introduction}
Crossed products of the form $C(X)\rtimes_r G$ arising from topological dynamical systems, say from $(X, G,
\alpha)$ for a countable discrete group $G$, an infinite compact metrizable space $X$ and a continuous action
$\alpha$, have long been an important source of examples and motivation for the classification of separable nuclear
$C^\ast$-algebras by the Elliott invariant, that is, ordered K-theory paired with traces. The origin of the classification programme of separable nuclear $C^*$-algebras by this invariant dates back to Elliott's work on the classification of AF-algebras \cite{El}. Then in the late 1980's, Elliott \cite{E} extended this result to the classification of A$\mathbb{T}$-algebras with real rank zero. We refer the reader to the survey papers \cite{Ro} and \cite{Win-ICM} for general background on the classification programme for separable nuclear $C^*$-algebras. Nowadays, in order to classify a certain
class of separable nuclear $C^\ast$-algebras it is often sufficient to show that the algebras in the class have certain regularity
properties such as finite nuclear dimension or $\mathcal{Z}$-stability. In 2008 Toms and Winter conjectured that
the three properties of strict comparison, finite nuclear dimension, and $\mathcal{Z}$-absorption are equivalent
for unital separable simple infinite-dimensional nuclear $C^\ast$-algebras (see \cite{W-Z}, for example). As a
result of work of several authors, this conjecture, known as the Toms-Winter conjecture, has been fully confirmed
under the hypothesis that the extreme tracial states form a compact set with finite covering dimension (see \cite{B-B-S-T-W-W}, \cite{K-R}, \cite{M-Sato}, \cite{M-S}, \cite{R}, \cite{S}, \cite{S-W-W}, \cite{T-W-W} \cite{Winter} and \cite{W}).

If we restrict to crossed products one would like to find properties of the dynamical system that will lead us to classification results. Motivated by this idea, one can
also formulate dynamical properties which parallel the three regularity properties in the $C^\ast$-setting of the Toms-Winter conjecture. A well-known dynamical analogue of strict comparison in  the $C^\ast$-setting is dynamical comparison, which was
discussed in \cite{B} and \cite{D}. An analogue of the nuclear dimension is the Rokhlin dimension introduced by Hirshberg, Winter and Zacharias \cite{H-W-Z} for actions of finite groups and $\mathbb{Z}$. Then Szab\'{o} \cite{Sz} and Szab\'{o}, Wu and Zacharias \cite{S-W-Z} generalized the definition of Rokhlin dimension for actions of $\mathbb{Z}^d$ and virtually nilpotent groups, respectively. On the other hand, the notions of amenability dimension and dynamical asymptotic dimension were introduced by Guentner, Willett and Yu \cite{G-W-Y}. Inspired by the work of Guentner, Willett and Yu, Kerr \cite{D} then introduced the tower dimension. All of these dimensions can be regarded as dynamical analogues of nuclear dimension because if one of above dimensions of an action is finite then the corresponding crossed product has finite nuclear dimension under the assumption that the covering dimension of the space $X$ is finite.

Kerr \cite{D} also introduced a notion called almost finiteness, which is an analogue of $\mathcal{Z}$-stability in the dynamical setting.  His definition in fact generalized the original notion defined in \cite{Matui} for a second countable \'{e}tale groupoid whose unit space is compact and totally disconnected. However,  it is proved in \cite{D} that this new definition of almost finiteness implies $\mathcal{Z}$-stability of the crossed product.  In addition, Kerr \cite{D} also considered the following triad of properties  for the free minimal system  $(X, G,
\alpha)$:
\begin{enumerate}[label=(\roman*)]
\item finite tower dimension;

\item almost finiteness;

\item dynamical comparison.
\end{enumerate}
Then results mentioned above suggest that these three notions indeed perform as a good dynamical analogues of
their Toms-Winter counterparts. Thus the relationship among them seems well worth further investigation not only on its own interest but also for the classification of crossed products.
Actually, it has been shown that (i)$\Rightarrow$(ii)$\Leftrightarrow$(iii) in \cite{D} in the case that the set
$E_G(X)$ of ergodic $G$-invariant Borel probability measures and covering dimension of $X$ are both finite. In \cite{D} it is also mentioned that the examples in \cite{G-K} satisfy (i) but not (ii).

Motivated by \cite{D}, we focus in the present paper on the relationship between (ii) and (iii). We show  that (ii)$\Leftrightarrow$(iii) still holds in the more general case that $E_G(X)$ is compact and zero-dimensional. We remark that (i)$\Rightarrow$(ii)$\Leftrightarrow$(iii) has been proved  without any assumption on $E_G(X)$ in ongoing work of Kerr and Szab\'{o} if the covering dimension of $X$ is finite.

Another important property in dynamical systems considered in this paper is the small boundary property.
The relationship between this property and the structure of $E_G(X)$ has been studied for a long time. It was
proved in \cite{L-W} and \cite{Shub-Weiss} that if the invariant ergodic probability Borel measure space
$E_{\mathbb{Z}}(X)$ of an action $\alpha: \mathbb{Z}\curvearrowright X$ is at most countable then the action has the small boundary property. However, it is a general open problem in dynamical systems that asks whether a minimal free action of an amenable group has the small boundary property if its space $M_G(X)$ of invariant Borel probability measures forms a Bauer simplex, that is, $E_G(X)=\partial_e M_G(X)$ is compact in the weak*-topology. In addition, the small boundary property also plays an important role in the
recent work of Elliott and Niu \cite{E-N} on the crossed products induced by minimal $\mathbb{Z}$-actions. It is
proved in \cite{E-N} that if such a $\mathbb{Z}$-action has the small boundary property then the crossed product is
$\mathcal{Z}$-stable. Motivated from these two perspectives it is worth investigating when minimal free actions have the small boundary property.  In this paper, we answer the problem mentioned above in the positive under an additional assumption that $E_G(X)$ is zero-dimensional. See Theorem 1.1 below. Our method unexpectedly involves a theorem of $C^\ast$-algebras due to Lin \cite{Lin} based on work of Cuntz and
Pedersen \cite{C-P}.

We write $\mathfrak{C}$ for the class of all stably finite infinite-dimensional unital simple separable nuclear $C^\ast$-algebras satisfying UCT and having finite nuclear dimension.  From the recent progress of classification programme we know that the class $\mathfrak{C}$ is classified by the Elliott invariant (Theorem 4.4 below).  The following are our main theorems.

\begin{theorem}
Let $G$ be a countable infinite discrete amenable group, $X$ an infinite compact metrizable space and $\alpha: G\curvearrowright X$ a minimal free continuous action of $G$ on $X$. Suppose that $E_G(X)$ is compact and zero-dimensional in the weak*-topology. Then $\alpha$ has the small boundary property.
\end{theorem}

\begin{theorem}
Let $G$ be a countable infinite discrete amenable group, $X$ an infinite compact metrizable space and $\alpha: G\curvearrowright X$ a minimal free continuous action of $G$ on $X$. Suppose that $E_G(X)$ is compact and zero-dimensional in the weak*-topology and $\alpha$ has dynamical $m$-comparison for some $m\in \mathbb{N}$. Then $\alpha$ is almost finite and thus the crossed product $C(X)\rtimes_r G$ is $\mathcal{Z}$-stable and belongs to the class $\mathfrak{C}$.
\end{theorem}

Combining Theorem 1.1 with Corollary 4.9 in \cite{E-N}, we have the following corollary. In this paper, however, instead of using Corollary 4.9 in \cite{E-N}, we directly verify that the crossed product under the assumption below has finite nuclear dimension and thus belongs to the class $\mathfrak{C}$.

\begin{corollary}
Let $X$ be an infinite compact metrizable space, and let $h: X\rightarrow X$ be a minimal homeomorphism. Suppose that $E_{\mathbb{Z}}(X)$ is compact and zero-dimensional in the weak*-topology. Then $C(X)\rtimes_r \mathbb{Z}$ belongs to the class $\mathfrak{C}$.
\end{corollary}

Our paper is organised as follows: Section 2 collects some preliminary results. In section 3 we study dynamical $m$-comparison, $m$-almost finiteness and prove Theorem 1.1 and the dynamical part of Theorem 1.2. In section 4, we discuss some properties of $C^\ast$-algebras arising from almost finite actions and finish the proof of Theorem 1.2.

\section{Preliminaries}
In this section, we recall some terminology and definitions used in the paper.
Throughout the paper $G$ denotes a countable infinite amenable group, $X$ denotes an infinite compact metrizable
topological space and $\alpha: G\curvearrowright X$ denotes a continuous action of $G$ on $X$.  We write $M(X)$ for the convex set of all regular Borel probability measures on $X$, which is a
weak* compact subset of $C(X)^\ast$. We write $M_G(X)$ for the convex set of $G$-invariant regular Borel
probability measures on $X$, which is a weak* compact subset of $M(X)$. We write $E_G(X)$ for the set of extreme
points of $M_G(X)$, which are precisely the ergodic measures in $M_G(X)$.

Given a unital $C^\ast$-algebra $A$, we write $T(A)$ for the convex set of all tracial states on $A$. Denote by
$B$ the reduced crossed product $C^\ast$-algebra $C(X)\rtimes_r G$ arising from a free action. For every measure $\mu$ in $M_G(X)$, the
function $\tau_{\mu}$ defined on $B$ by $\tau_{\mu}(a)=\int_X E(a)\,d\mu$ is a tracial state on $B$, where $E$ is
the canonical faithful conditional expectation from $B$ onto $C(X)$.  In the converse direction, every tracial state
induces an invariant measure on $X$ by restricting to $C(X)$. Actually Theorem 15.22 in \cite{NP} shows
that the function $H: M_G(X)\rightarrow T(B)$ defined by $H(\mu)=\tau_\mu$ is an affine bijection and it is not hard to see $H$ is actually an affine homeomorphism with
respect to the weak*-topology.   Therefore, we will usually identify the spaces $M_G(X)$ and $T(B)$. In addition,
$E_G(X)$ and $\partial_eT(B)$ correspond to each other under the same map.

The \emph{chromatic number} of a family $\mathcal{C}$ of subsets of a given set is defined to be the least $d\in \mathbb{N}$ such that there is a partition of $\mathcal{C}$ into $d$ subcollections each of which is disjoint. The idea of dynamical comparison dates back to Winter in 2012 and was discussed in the papers of Buck \cite{B} and Kerr \cite{D}.  We record the version appeared in \cite{D} here.

\begin{definition}(\cite[Definition 3.1]{D})
Let $F$ be a closed subset of $X$ and $O$ an open subset of $X$. We write $F\prec_m O$ if there exists a finite collection $\mathcal{U}$ of open subsets of $X$ which cover $F$, an $s_U\in G$ for each $U\in \mathcal{U}$, and a partition $\mathcal{U}=\bigsqcup_{i=0}^m\mathcal{U}_i$ such that for each $i=0,1,\dots ,m$ the images $s_UU$ for $U\in \mathcal{U}_i$ are pairwise disjoint subsets of $O$. When $m=0$ we also write $F\prec O$.
\end{definition}

\begin{definition}(\cite[Proposition 3.3]{D})
Let $m\in \mathbb{N}$. The action $\alpha: G\curvearrowright X$ is said to have \emph{dynamical $m$-comparison} (\textit{$m$-comparison} for short) if $F\prec_m O$ for every closed set $F\subset X$ and open set $O\subset X$ satisfying  $\mu(F)<\mu(O)$ for all $\mu\in M_G(X)$. When $m=0$, we will also say that the action has \emph{dynamical comparison} (\textit{comparison} for short).
\end{definition}

\begin{definition}(\cite[Definition 4.1]{D})
A \emph{tower} is a pair $(S,V)$ consisting of a subset $V$ of $X$ and a finite subset $S$ of $G$ such that the sets $sV$ for $s\in S$ are pairwise disjoint. The set $V$ is the\emph{ base} of the tower, the set $S$ is the \emph{shape} of the tower and the sets $sV$ for $s\in S$ are the levels of the tower.  We say that the tower $(S, V)$ is \textit{open} if $V$ is open. A finite collection of towers $\{(S_i,V_i): i\in I\}$ is called a \emph{castle} if    $S_iV_i\cap S_jV_j=\emptyset$ for all $i\neq j\in I$.
\end{definition}

\begin{definition}(\cite[Definition 11.2]{D})
Let $m\in \mathbb{N}$. We say that a free action $\alpha: G\curvearrowright X$ is $m$-\emph{almost finite} if for every $n\in \mathbb{N}$, finite set $K\subset G$, and $\delta>0$ there are a finite collection $\{(S_i, \overline{V_i}): i\in I\}$ of towers with following properties:
\begin{enumerate}[label=(\roman*)]
\item $V_i$ is an open subset of $X$ for every $i\in I$;

\item $S_i$ is $(K,\delta)$-invariant for every $i\in I$;

\item $\operatorname{diam}(s\overline{V_i})<\delta$ for every $i\in I$ and $s\in S_i$ and the family $\{S_i\overline{V_i}:i\in I\}$ has chromatic number at most $m+1$;

\item there are sets $S'_i\subset S_i$ for each $i\in I$ such that $|S'_i|\leq |S_i|/n$ and $X\setminus \bigsqcup_{i\in I}S_iV_i\prec \bigsqcup_{i\in I}S'_iV_i$.
\end{enumerate}
If $m=0$, we say  $\alpha: G\curvearrowright X$ is \emph{almost finite} for short. In this case $\{(S_i,V_i):i\in I\}$ is a castle.

\end{definition}

Note that the Definition 2.4 here seems to be stronger than the definition of almost finiteness in \cite{D} in which all towers are open. However, it can be shown that they are actually equivalent. We remark that it has been proved in \cite{D} that they are equivalent when $m=0$. In general, firstly we fix a metric $d$ on $X$. Given $n\in \mathbb{N}$, finite $K\subset G$, and $\delta>0$, suppose that we have an open castle $\{(S_i, V_i): i\in I\}$ satisfying the conditions of $m$-almost finiteness above. We start from condition (iv) and write $F$ for the set $X\setminus \bigsqcup_{i\in I}S_iV_i$ for simplicity. Since $F\prec \bigsqcup_{i\in I}S'_iV_i$ holds for the original castle,
there are open subset $O_1,\dots, O_n$ of $X$ and group elements $g_1,\dots, g_n\in G$ such that $F\subset \bigcup_{i=1}^n O_n$ and $\bigsqcup_{i=1}^n g_iO_i\subset\bigsqcup_{i\in
I}S'_iV_i$. An argument of partition of unity allows us to find open subset $W_j\subset O_j$  such that
$\overline{W_j}\subset O_j$ for each $j=1,2,\dots, n$ and $\{\overline{W_j}: j=1,2,\dots, n\}$ still form a cover of $F$. This allows us to find a $\delta>0$ to define a new open subset $O'_j=\{x\in X:
d(x, X\setminus O_j)>\delta\}$ such that $\overline{W_j}\subset O_j'\subset O_j$ for each $j=1,2,\dots, n$. This implies that  $\{O'_j: j=1,2,\dots, n\}$ also form a cover of $F$ and thus there is another $\delta'>0$ such that $B(F, \delta')=\{x\in X: d(x, F)<\delta'\}\subset \bigcup_{j=1}^n O'_j$. Then by the definition of $O_j'$ and the uniformly continuity of homeomorphisms induced by $g_1^{-1},\dots, g_n^{-1}$ there is a $\gamma>0$ such
that $d(g_jx, X\setminus g_jO_j)>\gamma$ for all $x\in O_j'$ and all $j=1,2,\dots, n$. Thus one has
$d(\bigsqcup_{j=1}^n g_jO'_j, X\setminus \bigsqcup_{i\in I}S'_iV_i)\geq d(\bigsqcup_{j=1}^n g_jO'_j, X\setminus
\bigsqcup_{j=1}^ng_jO_j)\geq \gamma$.

For an $\eta>0$ and an open set $U$ we write $U^{-\eta}=\{x\in X: d(x, X\setminus U)>\eta\}$ for the open subset
of $U$ shrunken by $\eta$. Observe that for each $i\in I$, there is a $\eta_i>0$ such that
$V_i\setminus V^{-\eta_i}_i\subset B(F, \delta')$. Then by the uniform continuity one can find an $\eta>0$ such that $X\setminus \bigsqcup_{i\in
I}S_iV^{-\eta}_i\subset B(F, \delta'/2)$ by shrinking all $\eta_i$ if necessary. In addition, by the same reason one can shrink $\eta$ furthermore such that $X\setminus \bigsqcup_{i\in I}S'_iV^{-\eta}_i\subset B(X\setminus \bigsqcup_{i\in
I}S'_iV_i, \gamma/2)$. This entails that
$X\setminus \bigsqcup_{i\in I}S_iV^{-\eta}_i\subset \bigcup_{j=1}^n O'_j$ while $\bigsqcup_{j=1}^n g_jO'_j\subset
\bigsqcup_{i\in I}S'_iV^{-\eta}_i$. This verifies that the new castle $\{(S_i, \overline{V^{-\eta}_i}): i\in I\}$ satisfies the
condition (iv). In addition we see that this new castle satisfies the other conditions of $m$-almost finiteness above trivially and thus $\{(S_i, \overline{V^{-\eta}_i}): i\in I\}$ is what we want. The converse direction is trivial.

We recall the notion of central sequence algebra. Let $A$ be a separable $C^\ast$-algebra. Set
\begin{align*}
A_{\infty}=\ell^\infty(\mathbb{N},A)/\{(a_n)_n\in\ell^\infty(\mathbb{N},A): \lim_{n\rightarrow \infty}\|a_n\|=0\}.
\end{align*}
We identify $A$ with the $C^\ast$-subalgebra of $A_{\infty}$ consisting of equivalence classes of constant sequences. We call $A_{\infty}\cap A'$ the \emph{central sequence algebra} of $A$, which consists of all equivalence classes whose representatives $(x_n)_n\in \ell^\infty(\mathbb{N},A)$ satisfy $\|[x_n,a]\|\rightarrow 0$ as $n\rightarrow \infty$ for all $a\in A$. Each such representing sequence $(x_n)_n$ is called a central sequence.

The following lemma is due to Lin \cite{Lin} based on work of Cuntz and Pedersen \cite{C-P}. This lemma enables us to realize strictly positive elements of $\operatorname{Aff}(T(A))$ via positive elements of $A$.

\begin{lemma}\emph{(\cite[Theorem 9.3]{Lin})}
Let A be a simple separable unital nuclear C*-algebra such that $T(A)\neq \emptyset$ and let $f$ be a strictly positive affine continuous function on $T(A)$. Then for any $\epsilon>0$, there exists $x\in A^+$ with $f(\tau)=\tau(x)$ for all $\tau\in T(A)$ and $\|x\|\leq \|f\|+\epsilon$.
\end{lemma}

The following lemma is due to Toms, White and Winter \cite{T-W-W}.

\begin{lemma}\emph{(\cite[Lemma 3.4]{T-W-W})}
Let $A$ be a separable unital C*-algebra with non-empty trace space $T(A)$. Let $T_0\subset T(A)$ be non-empty and suppose that $(e_n^{1})_n,\dots,(e_n^{L})_n$ are sequences of positive contractions in $A_+$ representing elements of $A_{\infty}\cap A'$ such that $\lim_{n\rightarrow \infty}\sup_{\tau\in T_0}|\tau(e_n^{(l)}e_n^{(l')})|=0$ for $l\neq l'$. Then there exist positive elements $\tilde{e}_n^{(l)}\leq e_n^{(l)}$ so that:
\begin{enumerate}[label=(\roman*)]
\item  $(\tilde{e}_n^{(l)})_n$ represents an element of $A_{\infty}\cap A'$;

\item  $\lim_{n\rightarrow \infty}\sup_{\tau\in T_0}|\tau(\tilde{e}_n^{(l)}-e_n^{(l)})|=0$;

\item $\tilde{e}_n^{(l)}\perp \tilde{e}_n^{(l')}$ in $A_{\infty}\cap A'$ for $l\neq l'$.
\end{enumerate}
\end{lemma}

\section{Dynamical comparison and almost finiteness}
In this section, we address the relationship between $m$-almost finiteness and dynamical $m$-comparison, establishing Theorem 1.1 and the dynamical part of Theorem 1.2. We first prove the following key lemma. This shows that for every finite disjoint collection of closed subsets of $E_G(X)$ we can find disjoint collections of closed subsets of $X$ that correspond to it in a nice way.
\begin{lemma}
Let $\alpha: G\curvearrowright X$ be a minimal free action such that $E_G(X)$ is compact in the weak*-topology. Then for every $\epsilon>0$ and set $W=\bigsqcup_{j=1}^L W_j$ which is a disjoint union of closed subsets of $E_G(X)$, there are pairwise disjoint compact subsets $\{K_j\}_{j=1}^L$ of $X$ such that $\mu(K_j)>1-\epsilon$ for all $\mu\in W_j$.
\end{lemma}
\begin{proof}

Given an $\epsilon>0$ and denote by $A$ the C*-algebra $C(X)\rtimes_r G$, which is simple since the action is minimal and free. We write  $H: M_G(X)\rightarrow T(A)$ for the homeomorphism defined by $\tau_\mu=H(\mu)$ such that $\tau_{\mu}(a)=\int_X E(a)\,d\mu$.  Note that $\partial_eT(A)=H(E_G(X))$ is compact under the weak*-topology. We also define $V_j=H(W_j)$ for all $ j=1,2, \dots ,L$, which are closed subsets of
$\partial_eT(A)$. For each $j=1,2,\dots,L$ and $n\in \mathbb{N}^+$, choose a strictly positive continuous function $f_n^j: \partial_eT(A)\rightarrow [0,1]$ with the norm $\|f_n^j\|=1+1/n$
such that $f_n^j=1+1/n$ on $V_j$ and $f_n^j=1/n$ on $\bigsqcup_{j'\neq j}V_{j'}$. This is possible by Urysohn's lemma as  the
$V_1, \dots, V_L$ are pairwise disjoint closed subsets of $\partial_eT(A)$. Since $\partial_eT(A)$ is
compact, for each $j$ and $n$ we can extend $f_n^j$ to a strictly positive continuous affine function on $T(A)$ with the same norm, which we also denote by $f_n^j$. Now apply Lemma 2.5 to obtain a sequence $(e_n^j)_n$ of positive elements of $A$ such that $\|e^j_n\|\leq 1+2/n$ and
$$f^j_n(\tau)=\tau(e_n^j)$$
for all $\tau\in T(A)$.
Define the functions $h_n^j=E(e_n^j)$ on $X$, where $E$ is the faithful conditional expectation from $A$ onto
$C(X)$. Observe that $\|h_n^j\|\leq \|e_n^j\|\leq 1+2/n$. Now, since $\tau(e_n^j)=\tau(h_n^j)$ for all $\tau\in T(A)$, we have
$$f^j_n(\tau)=\tau(h_n^j)$$
for all $\tau\in T(A)$. Define $g^j_n=\frac{h_n^j}{1+2/n}$. Then for each $n\in \mathbb{N}^+$ one has
$$\tau(g_n^j)=(n+1)/(n+2)$$ for every $\tau\in V_j$
while
$$\tau(g_n^j)=1/(n+2)$$ for every $\tau\in \bigsqcup_{j'\neq j}V_{j'}$.

Therefore, for the given $\epsilon$, for each $j=1,\dots, L$ there is an $N_j$ such that $\tau(g_n^j)>1-\epsilon$ whenever $\tau\in V_j$ and $n> N_j$. In addition, for $1\leq j, j'\leq L$ with $j\neq j'$ one has
$$\lim\limits_{n\to \infty}\sup\limits_{\tau\in \bigsqcup_{j=1}^L V_j}\tau(g_n^jg_n^{j'})=0.$$
Now apply Lemma 2.6 to the abelian $C^\ast$-algebra $C(X)$ with $T_0=\bigsqcup_{j=1}^L V_j$ and sequences $(g^1_n)_n$,\dots, $(g^L_n)_n$ (they are trivially central since $C(X)$ is abelian). Then we have sequences $(\tilde{g}^1_n)_n$,\dots, $(\tilde{g}^L_n)_n$ such that for each $1 \leq j\neq j'\leq L$ one has

\begin{enumerate}[label=(\roman*)]
\item $\tilde{g}_n^j\leq g_n^j$;
\item $\lim_{n\to \infty}\|\tilde{g}_n^j\tilde{g}_n^{j'}\|=0$;
\item $\lim_{n\rightarrow \infty}\sup_{\tau\in \bigsqcup_{j=1}^L V_j}|\tau(\tilde{g}_n^j-g_n^j)|=0$.
\end{enumerate}
Thus we may assume $\lim_{n\to \infty}\|g_n^jg_n^{j'}\|=0$ by replacing $g_n^j$ with $\tilde{g}_n^j$. Then for each pair $1 \leq j\neq j'\leq L$, there is an $M_{j,j'}\in \mathbb{N}$ such that $\|g_n^jg_n^{j'}\|<\epsilon^2$ whenever $n>M_{j,j'}$.

For the given $\epsilon>0$, choose an $n>\max\{N_j, M_{j, j'}: 1 \leq j\neq j'\leq L\}$ so that for all $j, j'=1,2,\dots,L$ and
$\tau\in V_j$ one has $\tau(g_n^j)>1-\epsilon$ and $\|g_n^jg_n^{j'}\|<\epsilon^2$ if $j\neq j'$. Define
$K_j=\{x\in X\colon g_n^j(x)\geq \epsilon\}$ for $j=1,2,\dots,L$. The sets $K_1,\dots, K_L$ are pairwise
disjoint since $x\in K_j\cap K_{j'}$ implies $g_n^j(x)g_n^{j'}(x)\geq \epsilon^2$, which is impossible. We write $U_j=\{x\in X\colon g_n^{j}(x)>0\}$ for $j=1,2,\dots,L$. Then for each $\mu\in W_j$ we have the inequality
$$\tau_\mu(g_n^j)=\int_X g_n^j\, d\mu=\int_{K_j}g_n^j\, d\mu+\int_{U_j\setminus K_j}g_n^j \, d\mu>1-\epsilon$$
while $\int_{U_j\setminus K_j}g_n^j\,d\mu\leq \epsilon\cdot\mu(U_j\setminus K_j)\leq \epsilon$. This implies that $\mu(K_j)=1\cdot\mu(K_j)\geq \int_{K_j}g_n^j\,d\mu>1-2\epsilon$.
\end{proof}

Note that for a fixed open subset $O$ of $X$, the function $f$ on $M(X)$ given by $f\colon \mu\rightarrow \mu(O)$ is lower semicontinuous. Similarly, if $F$ is closed, $f$ defined on $M(X)$ by $f\colon \mu\rightarrow \mu(F)$ is upper semicontinuous. The following lemma is a slightly stronger version of Lemma 9.1 in \cite{D}.

\begin{lemma}
Let $X$ be a compact metrizable space with a compatible metric $d$ and let $\Omega$ be a weak* closed subset of $M(X)$. Let $\lambda>0$. Let $A$ be a closed subset of $X$ such that $\mu(A)<\lambda$ for all $\mu\in \Omega$. Then there is a $\delta_0>0$ such that
$$\mu(\{x\in X\colon d(x,A)\leq \delta_0\})<\lambda$$
for all $\mu\in \Omega$.
\end{lemma}
\begin{proof}
For each $\delta>0$ set $N_\delta=\{x\in X: d(x,A)\leq \delta\}$. Then for every $\mu\in \Omega$, $\mu(A)<\lambda$ implies that there is a $\delta>0$ such that $\mu(N_\delta)<\lambda$. Now, write $O_\delta=\{\mu\in M(X): \mu(N_\delta)<\lambda\}$. Then $\{O_\delta: \delta>0\}$ is an open cover of $\Omega$ since $\mu(N_\delta)$ is an upper-semicontinuous function of $\mu$ as mentioned above. By the compactness of $\Omega$, one has $\Omega\subset \bigcup_{i=1}^n O_{\delta_i}$ for some subcover $\{O_{\delta_i}: i=1,2,\dots,n\}$. Let $\delta_0=\min\{\delta_i: i=1,2,\dots,n\}$. It follows that $\Omega\subset O_{\delta_0}$ and thus $\mu(N_{\delta_0})<\lambda$ for all $\mu\in \Omega$ .
\end{proof}

The following lemma allows us to adjust the collection of Borel towers arising in the Ornstein-Weiss tiling argument (Theorem 4.46 in \cite{Kerr-L}) to be a castle of a form that appears in the definition of  $m$-almost finiteness.

\begin{lemma}
Let $\alpha\colon G \curvearrowright X$ be a free action. Fix a $\mu\in M_G(X)$ and an integer $n\in \mathbb{N}$. For every finite subset $F\subset G$ and $\epsilon, \eta>0$, there is a castle $\{(T_k, \overline{V_k}): k=1,2,\dots,K\}$ such that for each $k$, $V_k$ is open, $T_k$ is $(F,\eta)$-invariant while $\operatorname{diam}(\overline{sV_k})<\eta$ for all $s\in T_k$, $\mu(\bigsqcup_{k=1}^K T_kV_k)>1-\epsilon$ and the interval $[\frac{1}{2n}|T_k|, \frac{1}{n}|T_k|]$ contains an integer $d_{k}$.
\end{lemma}
\begin{proof}
Since the action $\alpha\colon G \curvearrowright X$ is free, for all $x\in X$, one has $\mu(\{x\})=0$ and thus $\mu$ is atomless. Now, the Ornstein-Weiss theorem (Theorem 4.46 in \cite{Kerr-L}) implies that there is a castle $\{(T_k, B_k): k=1,2,\dots,K\}$ such that the shapes $T_k$ are $(F,\eta)$-invariant and the bases $B_k$ are Borel for all $k=1,2,\dots,K$ with $\mu(\bigsqcup_{k=1}^KT_kB_k)>1-\epsilon/2$. Since $G$ is infinite, we may enlarge $F$ and shrink $\eta$ sufficiently so that for each $k\leq K$ there is an integer $d_{k}$ in $[\frac{1}{2n}|T_k|, \frac{1}{n}|T_k|]$.

By uniform continuity, there is an $0<\eta'<\eta$ such that for all $s\in \bigcup_{k=1}^K T_k$ and $x,y\in X$, if $d(x,y)<\eta'$, then $d(sx,sy)<\eta$. For each $B_k$, there is an open cover of $\overline{B_k}$, say  $ \{O_{i,k} : i\in I_k\}$, such that diam$(O_{i,k})<\eta'/2$ for every $i\in I_k$. Then by compactness there is a finite subcover of $\overline{B_k}$, say $\overline{B_k}\subset \bigcup_{i=1}^{n_k} O_{i,k}$. Write $D_{i,k}=O_{i,k}\setminus\bigcup_{j=1}^{i-1}O_{j,k}$ and $C_{i,k}=B_k\cap D_{i,k}$, the latter of which satisfies diam$(C_{i,k})<\eta'/2$. Taking the sets $C_{i,k}$ now to be bases, we have a castle $\{(T_k, C_{i,k}):i=1,2,\dots,n_k, k=1,2,\dots,K\}$, which satisfies   $\mu(\bigsqcup_{k=1}^K\bigsqcup_{i=1}^{n_k}T_kC_{i,k})>1-\epsilon/2$. For each $i$ and $k$, there is a compact set $M_{i,k}\subset C_{i,k}$ such that $\mu(C_{i,k}\setminus M_{i,k})<\frac{\epsilon}{2\sum_{k=1}^Kn_k|T_k|}$ and hence $\mu(\bigsqcup_{k=1}^K\bigsqcup_{i=1}^{n_k}T_kM_{i,k})>1-\epsilon$.

We enlarge each $M_{i,k}$ to an open set $N_{i,k}$ such that diam$(N_{i,k})<\eta'$ and $\{(T_k,
N_{i,k}):i=1,2,\dots,n_k, k=1,2,\dots,K\}$ is a castle. To do this, by normality, for the disjoint family
$\{sM_{i,k}: s\in T_k, i\leq n_k, k\leq K\}$, we can first find another disjoint family  $\{U_{s,i,k}\supset
sM_{i,k}: s\in T_k, i\leq n_k, k\leq K\}$. Then for each $i\leq n_k$ and $k\leq K$, one can define
$N_{i,k}=\{x\in X: d(x, M_{i,k})<\eta'/2\}\cap (\bigcap_{s\in T_k} s^{-1}U_{s,i,k})$. Furthermore, for each pair
$(i,k)$, there is a $V_{i,k}$ such that $M_{i,k}\subset V_{i,k}\subset \overline{V_{i,k}}\subset N_{i,k}$.
The castle $\{(T_k, V_{i,k}):i=1,2,\dots,n_k, k=1,2,\dots,K\}$ is now the one that we want. Indeed,
diam$(\overline{V_{i,k}})<\eta'$ implies that diam$(\overline{sV_{i,k}})<\eta$  for all $s\in T_k$. Since
$M_{i,k}\subset V_{i,k}$, we have $\mu(\bigsqcup_{k=1}^K\bigsqcup_{i=1}^{n_k}T_kV_{i,k})>1-\epsilon$.

\end{proof}

Now we are ready to prove the following theorem, which may be regarded as a dynamical analogue of the known result on the Toms-Winter conjecture which states that strict comparison implies $\mathcal{Z}$-stability when the set of extreme tracial states is compact and finite-dimensional (\cite{K-R}, \cite{S} and \cite{T-W-W}).

\begin{theorem}
Let $\alpha\colon G \curvearrowright X$ be a minimal free action, where $E_G(X)$ is compact and of covering dimension $m$ in the weak*-topology. If $\alpha$ has dynamical comparison, then it is $m$-almost finite.
\end{theorem}
\begin{proof}
First we fix an integer $n\in \mathbb{N}$, a finite subset $F\subset G$, and real numbers $\eta>0$ and
$\frac{1}{4n+2}>\epsilon >0$. Then for every $\tau\in E_G(X)$, Lemma 3.3 implies that there is a castle $\mathcal{T}_{\tau}=\{(S_k,
\overline{V_k})\colon k=1,2,\dots,K\}$ where the sets $V_k$ are open, the shapes $S_k$  are
$(F,\eta)$-invariant,  diam($\overline{sV_k}$)$<\eta$ for all $s\in S_k$, $\tau(\bigsqcup_{k=1}^{K}
S_kV_k)>1-\epsilon$, and the interval $[\frac{1}{2n}|S_k|, \frac{1}{n}|S_k|]$ contains an integer $d_{k,\tau}$.
Define $T_{\tau}=\bigsqcup_{k=1}^{K} S_kV_k$, which is open. Then, by the remark above, the function on $E_G(X)$ defined by $\rho\rightarrow \rho(T_{\tau})$ is lower semicontinuous.

For every $\tau\in E_G(X)$, we define the open neighborhood $U_{\tau}=\{\rho\in E_G(X)\colon \rho(T_\tau)>1-\epsilon\}$ of $\tau$, which is open by the semicontinuity of $\rho(T_\tau)$. The compactness of $ E_G(X)$ then implies that there is an $I\in \mathbb{N}$ such that $ E_G(X)=\bigcup_{i=1}^I U_{\tau_i}$. Since $\dim(E_G(X))\leq m$, there is a finite cover $\mathcal{W}$ of $E_G(X)$ consisting of closed sets such that $\mathcal{W}$ refines $\mathcal{U}=\{U_{\tau_1},\dots,U_{\tau_I}\}$ and a map $c: \mathcal{W}\rightarrow  \{0,1,\dots,m\}$ such that $c(W)=c(W')$ implies $W\cap W'=\emptyset$. For each $i\in \{0,1,\dots,m\}$, write $\mathcal{W}^{(i)}=\{W_1^{(i)},\dots,W_{L_i}^{(i)}\}$. Then for each $i\leq m$ and $j\leq L_i$, there is a $\tau^{(i)}_{j}$ such that $W_j^{(i)}\subset U_{\tau_j^{(i)}}$. This implies that there is a finite collection of towers $\{(S^{(i)}_{k,j}, \overline{V^{(i)}_{k,j}})\colon k=1,2,\dots, K_j^{(i)}, j=1,2,\dots,L_i,i=0,\dots,m\}$ such that for each $\rho\in W^{(i)}_j$ one has $\rho(T_{\tau^{(i)}_j})=\rho(\bigsqcup_{k=1}^{K_j^{(i)}}S^{(i)}_{k,j}V^{(i)}_{k,j})>1-\epsilon$.

Now fix a $i\in \{0,1,\dots,m\}$.   Apply Lemma 3.1 to $R_i=\bigsqcup_{j=1}^{L_i}W_j^{(i)}$ to obtain a collection of pairwise disjoint compact sets $\{C^{(i)}_j\}_{j=1}^{L_i}$ such that for all $\rho\in W_j^{(i)}$ one has $\rho(C_j^{(i)})>1-\frac{\epsilon}{(\sum_{k=1}^{K^{(i)}_j}|S^{(i)}_{k,j}|)^2}$. For $\{C^{(i)}_j\}_{j=1}^{L_i}$, there are collections of pairwise disjoint open sets $\{N^{(i)}_j\}_{j=1}^{L_i}$ and $\{M^{(i)}_j\}_{j=1}^{L_i}$ such that $C^{(i)}_j\subset N^{(i)}_j\subset \overline{N^{(i)}_j}\subset M^{(i)}_j$. Define $Y^{(i)}_j=\bigcap_{s\in \bigcup_{k=1}^{K^{(i)}_j} S^{(i)}_{k,j}}s^{-1}N^{(i)}_j$.

Note that towers in the collection $\{(S^{(i)}_{k,j}, \overline{V^{(i)}_{k,j}\cap Y^{(i)}_j})\colon k=1,2,\dots, K_j^{(i)}, j=1,2,\dots,L_i\}$  are pairwise disjoint. Indeed, for all $j, j'\leq L_i$, $s\in S^{(i)}_{k_1,j}$ and $t\in S^{(i)}_{k_2,j'}$ one has $s(\overline{V^{(i)}_{k_1,j}\cap Y^{(i)}_j})\subset \overline{N^{(i)}_j}$ and $t(\overline{V^{(i)}_{k_2,j'}\cap Y^{(i)}_{j'}})\subset \overline{N^{(i)}_{j'}}$. Then for all $\rho\in W^{(i)}_j$:

\begin{align*}
 \rho((Y^{(i)}_j)^c)&=\rho(\bigcup_{s\in \bigcup_{k=1}^{K^{(i)}_j} S^{(i)}_{k,j}}s^{-1}(N^{(i)}_j)^c)
 \leq\sum_{k=1}^{K^{(i)}_j}|S^{(i)}_{k,j}|\cdot\frac{\epsilon}{(\sum_{k=1}^{K^{(i)}_j}|S^{(i)}_{k,j}|)^2}\\
                     &=\frac{\epsilon}{(\sum_{k=1}^{K^{(i)}_j}|S^{(i)}_{k,j}|)}.\\
\end{align*}

 It follows that $$\rho(V^{(i)}_{k,j}\cap Y^{(i)}_j)\geq \rho(V^{(i)}_{k,j})-\frac{\epsilon}{(\sum_{k=1}^{K^{(i)}_j}|S^{(i)}_{k,j}|)},$$
 and thus
\begin{align*}
\rho(\bigsqcup_{k=1}^{K^{(i)}_j} S^{(i)}_{k,j}(V^{(i)}_{k,j}\cap Y^{(i)}_j))&\geq \sum_{k=1}^{K^{(i)}_j}|S^{(i)}_{k,j}|(\rho(V^{(i)}_{k,j})-\frac{\epsilon}{(\sum_{k=1}^{K^{(i)}_j}|S^{(i)}_{k,j}|)})\\
 &=\sum_{k=1}^{K^{(i)}_j}|S^{(i)}_{k,j}|\rho(V^{(i)}_{k,j})-\sum_{k=1}^{K^{(i)}_j}|S^{(i)}_{k,j}|\frac{\epsilon}{(\sum_{k=1}^{K^{(i)}_j}|S^{(i)}_{k,j}|)}\\
 &=\rho(\bigsqcup_{k=1}^{K_j^{(i)}} S^{(i)}_{k,j}V^{(i)}_{k,j})-\epsilon\\
 &\geq 1-2\epsilon
\end{align*}
for all $\rho\in W^{(i)}_j$.

Then, since $ E_G(X)=\bigcup_{i=0}^m R_i=\bigcup_{i=0}^m\bigsqcup_{j=1}^{L_i} W^{(i)}_j$, for all $\rho\in E_G(X)$ one has:
$$(\star)\ \ \ \ \rho(\bigcup_{i=0}^m\bigsqcup_{j=1}^{L_i}\bigsqcup_k^{K^{(i)}_j} S^{(i)}_{k,j}(V^{(i)}_{k,j}\cap Y^{(i)}_j))\geq 1-2\epsilon.$$
Define $O=\bigcup_{i=0}^m\bigsqcup_j^{L_i}\bigsqcup_k^{K^{(i)}_j} S'^{(i)}_{k,j}(V^{(i)}_{k,j}\cap Y^{(i)}_j)$ where $S'^{(i)}_{k,j}\subset S^{(i)}_{k,j}$ with $|S'^{(i)}_{k,j}|=d_{k,\tau^{(i)}_j}\in[\frac{1}{2n}|S^{(i)}_{k,j}|, \frac{1}{n}|S^{(i)}_{k,j}|] $ and $F=X\setminus\bigcup_{i=0}^m\bigsqcup_j^{L_i}\bigsqcup_k^{K^{(i)}_j} S^{(i)}_{k,j}(V^{(i)}_{k,j}\cap Y^{(i)}_j)$. This implies that $\rho(O)\geq \frac{1}{2n}(1-2\epsilon)$ and $\rho(F)<2\epsilon$ for all $\rho\in  E_G(X)$. Applying Lemma 3.2 to $F$, there is an open set $U\supset F$ such that $\rho(U)<2\epsilon$ for all $\rho\in  E_G(X)$. In the same manner, applying Lemma 3.2 to $O^c$, there is a closed set $D\subset O$ such that $\rho(D)\geq\frac{1}{2n}(1-2\epsilon)$ for all $\rho\in  E_G(X)$. Then since our $\epsilon$ is chosen to be less than $\frac{1}{4n+2}$, one has $\frac{1}{2n}(1-2\epsilon)>2\epsilon$. It turns out that for every $\rho\in  E_G(X)$ one has:

$$(\blacklozenge)\ \ \  \rho(D)\geq \frac{1}{2n}(1-2\epsilon)>2\epsilon>\rho(U);$$

By convexity, $(\blacklozenge)$ also holds for all $\rho\in \textrm{conv}( E_G(X))$. Now, let $\tau_n\rightarrow \tau$ where $\tau_n\in \textrm{conv}( E_G(X))$ and $\tau\in M_G(X)$.  By the portmanteau theorem (Theorem 17.20 in \cite{Kech}), $\tau(D)\geq \limsup_{n\rightarrow \infty}\tau_n(D)\geq \frac{1}{2n}(1-2\epsilon)$ and  $\tau(U)\leq \liminf_{n\rightarrow \infty}\tau_n(U)\leq 2\epsilon$, which implies that  $\rho(D)\geq \frac{1}{2n}(1-2\epsilon)>2\epsilon\geq\rho(U)$ holds for all $\tau\in M_G(X)$. Therefore, $\tau(O)\geq \tau(D)>\tau(U)\geq \tau(F)$ for all $\tau\in M_G(X)$.

Therefore, since the action $\alpha$ has dynamical comparison, one has:
\begin{align*}
X\setminus\bigcup_{i=0}^m\bigsqcup_j^{L_i}\bigsqcup_k^{K^{(i)}_j} S^{(i)}_{k,j}(V^{(i)}_{k,j}\cap Y^{(i)}_j)\prec \bigcup_{i=0}^m\bigsqcup_j^{L_i}\bigsqcup_k^{K^{(i)}_j} S'^{(i)}_{k,j}(V^{(i)}_{k,j}\cap Y^{(i)}_j).
\end{align*}
Finally, we write $\mathcal{T}_i$ for the collection of towers $\{(S^{(i)}_{k,j}, \overline{V^{(i)}_{k,j}\cap Y^{(i)}_j})\colon k=1,2,\dots, K^{(i)}_j, j=1,2,\dots,L_i\}$ for $i=0,1,\dots,m$. Observe that towers in each $\mathcal{T}_i$ are pairwise disjoint.  This implies that the collection of towers $\{\mathcal{T}_i: i=0,1,\dots,m\}$ witnesses that $\alpha$ is $m$-almost finite.

\end{proof}

The theorem below arises from the one above if we assume $E_G(X)$ is compact and zero-dimensional, but weaken ``comparison'' to ``m-comparison'' in order to arrive at almost finiteness. The idea of the proof of the following theorem comes from Theorem 9.2 in \cite{D}. This completes the proof of the dynamical part of Theorem 1.2.

\begin{theorem}
Let $\alpha\colon G \curvearrowright X$ be a minimal free action such that $E_G(X)$ is compact and zero-dimensional in the weak*-topology. If $\alpha$ has dynamical $m$-comparison for some $m\in \mathbb{N}$, then it is almost finite.
\end{theorem}
\begin{proof}
First,  we fix $n\in \mathbb{N}$, a finite set $F\subset G$, $\eta>0$ and $\frac{1}{4(m+1)n+2}>\epsilon >0$. Then by the same proof of Theorem 3.4, there exists a castle $\{(S_i, \overline{V_i}): i\in I \}$ where the sets $V_i$ are open, the shapes $S_i$  are $(F,\eta)$-invariant,  diam($\overline{sV_i}$)$<\eta$ for all $s\in S_i$ and $\mu(\bigsqcup_{i\in I}S_iV_i)\geq1-2\epsilon$ for all $\mu\in M_G(X)$. In addition, since $G$ is infinite we can enlarge $F$ to make all $S_i$ have large enough cardinality so that there is an $S'_{i,0}\subset S_i$ satisfying $\frac{1}{2(m+1)n}|S_i|<|S'_{i,0}|<\frac{1}{(m+1)n}|S_i|$. Write $O=\bigsqcup_{i\in I}S'_{i,0}V_i$ and $F=X\setminus \bigsqcup_{i\in I}S_iV_i$. Then we have the following inequality for all $\mu\in M_G(X)$:
$$\mu(O)\geq\frac{1}{2(m+1)n}(1-2\epsilon)>2\epsilon\geq \mu(F).$$
Since $\alpha$ has $m$-comparison, there is a finite collection $\mathcal{U}$ of open subsets of $X$ which cover $F$, an $s_U\in G$ for each $U\in \mathcal{U}$, and a partition $\mathcal{U}=\bigsqcup_{j=0}^m\mathcal{U}_j$ such that for each $j=0,1,\dots,m$ the images $s_UU$ for $U\in \mathcal{U}_j$ are pairwise disjoint subsets of $O$. For each $i\in I$, since $|S'_{i,0}|<\frac{1}{(m+1)n}|S_i|$, we can choose pairwise disjoint sets $S'_{i,k}$ of the same cardinality, for $k=1,2, \dots,m$, which allows us to choose a bijection $\varphi_{i,j}: S'_{i,0}\rightarrow S'_{i,j}$.

For $U\in \mathcal{U}$, $i\in I$ and $t\in S'_{i,0}$ we denote by $W_{U,i,t}$ the open set $U\cap s_U^{-1}tV_i$.
For each $j\in\{1,2,\dots,m\}$ and $U\in \mathcal{U}_{j}$, the family $\{W_{U,i,t}:i\in I, t\in S'_{i,0}\}$
forms a partition of $U$. This implies that the sets $\varphi_{i,j}(s_U)t^{-1}s_UW_{U,i,t}$  for  $U\in
\mathcal{U}_{j},i\in I,t\in S'_{i,0}$ are pairwise disjoint and contained in $\bigsqcup_{i\in I}S'_{i,j}V_i$.
This entails  $F\prec \bigsqcup_{i\in I}S'_iV_i$ where $S'_i=\bigsqcup_{j=0}^m S'_{i,j}$ with
$|S'_i|<\frac{1}{n}|S_i|$ and thus verifies that $\alpha$ is almost finite.

\end{proof}

Combined with (i)$\Rightarrow$(ii)$\Rightarrow$(iii)$\Rightarrow$(iv) in Theorem 9.2 in \cite{D}, the theorem above yields the same conclusion as  this theorem from \cite{D} under a weaker hypothesis.

\begin{corollary}
Let $\alpha: G \curvearrowright X$ be a minimal free action. If $E_G(X)$ is compact and zero-dimensional, the following conditions are equivalent.
\begin{enumerate}[label=(\roman*)]
\item $\alpha$ is almost finite;

\item $\alpha$ is $m$-almost finite for some $m\geq 0$;

\item $\alpha$ has comparison;

\item  $\alpha$ has $m$-comparison for some $m\geq 0$.
\end{enumerate}
\end{corollary}

Now, we would like to bring the small boundary property into the picture. The small boundary property was introduced in \cite{L-W}. The small boundary property is equivalent to mean dimension zero for $\mathbb{Z}$-actions \cite{Linden} and $\mathbb{Z}^d$-actions \cite{G-L-T} with marker property and thus in particular this equivalence holds if the action is minimal and free.  However it is still open for actions of a general amenable group that whether mean dimension zero and the small boundary property are equivalent.

\begin{definition}
An action $\alpha: G\curvearrowright X$ is said to have the \emph{small boundary property} if for every point $x\in X$ and every open $U\ni x$ there is an open neighborhood $V\subset U$ of $x$ such that $\mu(\partial V)=0$ for every $\mu\in M_G(X)$.
\end{definition}

In \cite{L-W} and \cite{Shub-Weiss}, a cardinality argument is used to show that if $E_{\mathbb{Z}}(X)$ is at most countable, then   $\alpha: \mathbb{Z}\curvearrowright X$ has the small boundary property. Theorem 1.1 thus provides a generalization of this result under the assumption that the action is minimal and free. The following proposition was communicated to me by G\'{a}bor Szab\'{o}.

\begin{proposition}
Let $\alpha: G\curvearrowright X$. Suppose that for every $\delta>0$, $\epsilon>0$ there is a collection $\mathcal{U}$ of pairwise disjoint open sets such that $\max_{U\in \mathcal{U}}\textrm{diam}(U)<\delta$ and $\mu(X\setminus \bigcup\mathcal{U})<\epsilon$ for all $\mu\in M_G(X)$. Then $\alpha: G\curvearrowright X$ has the small boundary property.
\end{proposition}
\begin{proof}
Fix a metric $d$ on the space $X$. We firstly claim that given $F\subset O$ where $F$ is closed and $O$ is open, for every $\epsilon>0$ there is an open neighbourhood $V$ of $F$ such that $F\subset V\subset \overline{V}\subset O$ and $\mu(\partial V)<\epsilon$.

To show this claim firstly observe that $l=d(F, O^c)>0$, which implies that $F\subset\overline{B}(F, l/2)\subset B(F, l)\subset O$ where $\overline{B}(F, l/2)$ is defined to the set $\{x\in X: d(x, F)\leq l/2\}$ while $B(F, l)$ is defined to be the set $\{x\in X: d(x, F)<l\}$. Now, for the number $l/2$ and a given positive number $\epsilon>0$ one can find a collection $\mathcal{U}$ of pairwise disjoint open sets such that $\max_{U\in \mathcal{U}}\textrm{diam}(U)<l/2$ and $\mu(X\setminus \bigcup\mathcal{U})<\epsilon$ for all $\mu\in M_G(X)$.

Now define $K=F\setminus\bigcup\mathcal{U}\subset X\setminus\bigcup\mathcal{U}$ which entails that $\mu(K)<\epsilon$ for all $\mu\in M_G(X)$. Then Lemma 3.2 implies that there is an open subset $M$ such that $K\subset M\subset \overline{M}\subset O$ such that $\mu(\overline{M})<\epsilon$ for all $\mu\in M_G(X)$. Now consider $\{U\in \mathcal{U}: F\cap U\neq \emptyset\}\cup \{M\}$ form an open cover of $F$ and thus has a finite subcover, say, $\{U_1,\dots, U_n, M\}$ by compactness.  For each $i=1,\dots, n$ since $\operatorname{diam}(U_i)<l/2$ and $U_i\cap F\neq \emptyset$, one has $U_i\subset B(F,l/2)$ and thus $\overline{U_i}\subset \overline{B}(F, l/2)\subset B(F,l)\subset O$. Now define $V=(\bigsqcup_{i=1}^n U_i)\cup M$, which satisfies that $F\subset V\subset \overline{V}\subset O$.

In addition, consider $\partial V\subset \bigcup_{i=1}^n \partial U_i\cup \partial M$. Since the family $\mathcal{U}$ is disjoint, each $\partial U_i\subset X\setminus \bigcup\mathcal{U}$ and thus $\bigcup_{i=1}^n \partial U_i\subset X\setminus \bigcup\mathcal{U}$. Combining with the fact $\partial M\subset \overline{M}$, one has $\mu(\partial V)<2\epsilon$ for all $\mu\in M_G(X)$. This completes the claim.

Now, let $x\in O$ where $x\in X$ and $O$ is an open subset of $X$. Then we proceed by induction to construct
sequences $x\in U_1\subset U_2\subset\dots\subset O$ and $O\supset Z_1\supset Z_2\supset\dots $ such that $\partial U_n\subset Z_n$
and $\mu(\overline{Z_n})<1/n$ for all $\mu\in M_G(X)$ and $n\in \mathbb{N}^+$.  Firstly, the claim above allows us to choose an open neighbourhood $U_1$ of $x$ such that $x\in U_1\subset
\overline{U_1}\subset O$ such that $\mu(\partial U_1)<1$ for all $\mu\in M_G(X)$. Then apply Lemma 3.2 to
$\partial U_1$ to obtain an open neighbourhood $Z_1$ of $\partial U_1$ such that $\overline{Z_1}\subset O$ and $\mu(\overline{Z_1})<1$ for all $\mu\in M_G(X)$. Suppose that we
have constructed $U_1\subset U_2\subset \dots \subset U_k\subset O$ and  $O\supset Z_1\supset Z_2\supset\dots
\supset Z_k$ such that $\partial U_n\subset Z_n$ and $\mu(\overline{Z_n})<1/n$ for all $\mu\in M_G(X)$ and $n=1,\dots, k$.
Now we define $U_{k+1}$ and $Z_{k+1}$. Apply the claim above to $\overline{U_k}\subset U_k\cup Z_k$ then there is an open subset $U_{k+1}$ such that $\overline{U_k}\subset U_{k+1}\subset \overline{U_{k+1}}\subset U_k\cup
Z_k\subset O$ with $\mu(\partial U_{k+1})<1/(k+1)$ for all $\mu\in M_G(X)$. Observe that $Z_k$ is an open
neighbourhood of $\partial U_{k+1}$. Then by Lemma 3.2 again there is an open subset $Z_{k+1}$ such that $\partial
U_{k+1}\subset Z_{k+1}\subset \overline{Z_{k+1}}\subset Z_k$ and $\mu(\overline{Z_{k+1}})<1/(k+1)$ for all
$\mu\in M_G(X)$. This finishes our construction.

Now define $U=\bigcup_{n=1}^\infty U_n$. Then $x\in U\subset O$. In addition, our construction implies that $U_k\subset U\subset U_k\cup Z_k$ for each $k\in \mathbb{N}^+$. Therefore one has
\[\partial U=\overline{U}\setminus U\subset \overline{U_k\cup Z_k}\setminus U_k\subset \overline{Z_k}\]
for each $k\in \mathbb{N}^+$. This entails that $\mu(\partial U)=0$ for all $\mu\in M_G(X)$. This verifies the small boundary property.
\end{proof}

We remark that the converse of the proposition above is also true. But the direction in the proposition above is good enough for our purpose to prove Theorem 1.1:
\begin{proof}(Theorem 1.1)
Let $\alpha: G\curvearrowright X$ be a minimal free action. We revisit the proof of Theorem 3.4. Given a finite set $F\subset G$, $\epsilon>0$ and $\delta>0$, if $E_G(X)$ is compact and zero-dimensional, then the process allows us to construct a finite open castle $\{(T_i, V_i): i\in I\}$ such that
\begin{enumerate}[label=(\roman*)]
\item $T_i$ is $(F,\delta)$-invariant for every $i\in I$;

\item $\textrm{diam}(tV_i)<\delta$ for all $t\in T_i$ and all $i\in I$;

\item $\mu(X\setminus\bigsqcup_{i\in I}T_iV_i)<\epsilon$ for all $\mu\in E_G(X)$ (i.e. inequality ($\star$) ).
\end{enumerate}
Then, the same argument as in the proof of Theorem 3.4, together with Lemma 3.2 and the portmanteau theorem, imply that:

(iii')\ \ \  $\mu(X\setminus\bigsqcup_{i\in I}T_iV_i)<\epsilon$ for all $\mu\in M_G(X).$

At last, Proposition 3.8 implies that  $\alpha: G\curvearrowright X$ has the small boundary property.
\end{proof}

We close this section by remarking that the property that requires the existence of castles satisfying properties (i), (ii) and (iii') is called ``almost finiteness in measure'' and was introduced in ongoing work of Kerr and Szab\'{o} in which it is proved
that a minimal free action $\alpha: G\curvearrowright X$ has the small boundary property if and only if it is almost finite in measure.

\section{Crossed product $C^\ast$-algebras arising from almost finite actions}
In this section, we explore some properties of the crossed products arising from a minimal free almost finite action $\alpha: G\curvearrowright X$. The following is one of the most important justifications, from the $C^\ast$-algebra perspective, for the study of the dynamical analogue of the Toms-Winter conjecture.

\begin{theorem}\emph{(\cite[Theorem 12.4]{D})}
Let $\alpha: G\curvearrowright X$ be a minimal free action. If $\alpha$ is almost finite, then the crossed product $C(X)\rtimes_rG$ is $\mathcal{Z}$-stable.
\end{theorem}

We observe that any crossed product $C^\ast$-algebra $A=C(X)\rtimes_rG$ arising from a minimal action $\alpha: G\curvearrowright X$ is stably finite since $\tau(a)=\int_X E(a)\,d\mu$ is a faithful tracial state on $A$, where $\mu$ is an invariant probability measure on $X$ (such a $G$-invariant probability measure always exists since the group $G$ is assumed to be amenable) and $E$ is the canonical faithful conditional expectation from $A$ to $C(X)$. Therefore, if the action $\alpha$ is also free and almost finite, then $A=C(X)\rtimes_rG$  has stable rank one by Theorem 4.1 above and Theorem 6.7 in \cite{R}. We remark that both Kerr \cite{D} and Suzuki \cite{Suzuki} generalize the notion ``almost finiteness'' from \cite{Matui}. Both generalizations coincide with the original one if the space $X$ is the Cantor set.  They differ in general since ``almost finiteness'' in \cite{Suzuki} does not necessarily imply $\mathcal{Z}$-stability.

Compared with stable rank, it is much harder to determine the real rank as well as the tracial rank of a $C^\ast$-algebra arising from minimal free almost finite actions of an infinite amenable group. The following result is due to R{\o}rdam.

\begin{theorem}\emph{(\cite[Theorem 7.2]{R})}
The following conditions are equivalent for each unital, simple, exact,  finite, $\mathcal{Z}$-absorbing $C^\ast$-algebra $A$.

\emph{(i)} $rr(A)=0$;

\emph{(ii)} $\rho(K_0(A))$ is uniformly dense in $\emph{Aff}(T(A))$\newline
where $\rho$ is defined by $\rho(g)(\tau)=K_0(\tau)(g)$.
\end{theorem}

A crossed product $C^\ast$-algebra $A=C(X)\rtimes_rG$ arising from minimal free almost finite actions of an infinite
amenable group certainly satisfies the assumption of the theorem above. However, it is generally very difficult
to verify whether $A$ satisfies condition (ii) in the theorem above. Known examples are the irrational
rotation algebras, which are included in a collection of more general examples constructed by  Lin and Phillips
in \cite{L-P}. Note that every irrational rotation on $\mathbb{T}$ is indeed almost finite by Theorem 1.2 since
it is  uniquely ergodic and has dynamical comparison (see \cite{B}). It is worth mentioning that the result of Lin and Phillips
in fact recovers the Elliott-Evans Theorem \cite{E-E} stating that every irrational rotation algebra is an
A$\mathbb{T}$-algebra with real rank zero. On the other hand, if the space $X$ is the Cantor set, Phillips \cite{Phillips} worked on almost AF Cantor groupoids and proved that the crossed product arising from a minimal free action $\mathbb{Z}^d \curvearrowright X$ has real rank zero.  Suzuki \cite{Suzuki} then generalized the result of Phillips by a different approach by proving the following theorem in \cite{Suzuki}.

\begin{theorem}\emph{(\cite[Remark 4.3]{Suzuki})}
Let $\alpha: G\curvearrowright X$ where $X$ is the Cantor set. If $\alpha$ is almost finite, then  the crossed product $C(X)\rtimes_r G$ has real rank zero.

\end{theorem}

Suzuki \cite{Suzuki} also proved that $\alpha: G\curvearrowright X$ is almost finite if $G$ is abelian and  $X$ is the Cantor set. Then, as an application of Theorem 4.3,  $C(X)\rtimes_r G$ has real rank zero if $G$ is abelian and $X$ is the Cantor set.

We close this section by finishing proving Theorem 1.2 and Corollary 1.3. The following remarkable classification theorem comes from a combination of works of Elliott-Gong-Lin-Niu \cite{E-G-L-N}, Gong-Lin-Niu \cite{G-L-N} and  Tikuisis-White-Winter \cite{Ti-W-W}.

\begin{theorem}
The class $\mathfrak{C}$ of infinite-dimensional stably finite simple separable unital $C^\ast$-algebras satisfying the UCT and having finite nuclear dimension is classified by the Elliott invariant.
\end{theorem}

This leads us to the completion of the proof of Theorem 1.2 and Corollary 1.3.

\begin{proof}(Theorem 1.2)
Since $E_G(X)$ is compact and zero-dimensional, Theorem 3.5 and Theorem 4.1 imply that the crossed product $A=C(X)\rtimes_r G$ is $\mathcal{Z}$-stable. Then the result in \cite{B-B-S-T-W-W} shows that $A$ has finite nuclear dimension because the extreme tracial states form a compact set in the weak*-topology. In addition, $C(X)\rtimes_r G$ is isomorphic to a $C^\ast$-algebra of a Hausdorff,
locally compact, second countable amenable transformation groupoid and thus satisfies UCT by a result of Tu \cite{Tu}. Then the crossed product $A=C(X)\rtimes_r G$ belongs to the class $\mathfrak{C}$, which is classified by the Elliott invariant by Theorem 4.4 above.
\end{proof}

\begin{proof}(Corollary 1.3)
Suppose $E_{\mathbb{Z}}(X)$ is compact and zero-dimensional, then $\alpha: \mathbb{Z}\curvearrowright X$ has the small boundary property by Theorem 1.1. Then \cite{E-N} implies that $A=C(X)\rtimes_r \mathbb{Z}$ is $\mathcal{Z}$-stable and therefore $A$ has finite nuclear dimension. In addition, the result of Tu \cite{Tu} shows that $A$ satisfying UCT as mentioned above. Then the crossed product $A=C(X)\rtimes_r \mathbb{Z}$ belongs to the class $\mathfrak{C}$, which is classified by the Elliott invariant by Theorem 4.4 above.
\end{proof}

\section{Acknowledgement}
The author should like to thank his supervisor David Kerr for a lot of very inspiring suggestions, helpful discussions and corrections. He thanks G\'{a}bor Szab\'{o} for letting him know Proposition 3.8. He also thanks Yuhei Suzuki for pointing out an incorrect citation occurred in the first version. Finally, he should like to thank the anonymous referee whose comments and suggestions helped to greatly improve the paper.


\begin{thebibliography}{10}
\bibitem{B-B-S-T-W-W}J. Bosa, N. Brown, Y. Sato, A. Tikuisis, S. White and W. Winter. Covering dimension of $C^\ast$-algebras and 2-coloured classification. To appear in \textit{Mem. Amer. Math. Soc.}.

\bibitem{B-Ozawa}N. Brown and N. Ozawa. \textit{C*-algebras and finite-dimensional approximations}.
Graduate Studies in Mathematics, 88. American Mathematical Society, 2008, 509 pp.

\bibitem{B}J. Buck. Smallness and comparison properties for minimal dynamical systems. arXiv:1306.6681v1.

\bibitem{C-P}J. Cuntz and G. K. Pederson. Equivalence and traces on $C^\ast$-algebras.  \textit{J. Funct. Anal.} \textbf{33} no.2 (1979), 135-164.

\bibitem{E}G. A. Elliott. On the classification of $C^*$-algebras of real rank zero. \textit{ J. Reine. Angew. Math.} \textbf{443} (1993), 179-219.

\bibitem{El}G. A. Elliott. On the classification of inductive limits of sequences of semisimple finite-dimensional algebras. \textit{J. Algebra} \textbf{38} (1976), 29-44.

\bibitem{E-E}G. A. Elliott and D. E. Evans. The structure of the irrational rotation algebra. \textit{Ann. Math.} (2) \textbf{138} (1993), 477¨C501.

\bibitem{E-G-L-N} G. Elliott, G. Gong, H. Lin, and Z. Niu. On the classification of simple amenable $C^\ast$-algebras with finite decomposition rank, II. arXiv:1507.03437.

\bibitem{E-N}G. A. Elliott and Z. Niu. The $C^\ast$-algebra of a minimal homeomorphism of zero mean dimension. To appear in \textit{Duke Mathematical Jounal}.

\bibitem{G-L-T}Y. Gutman, E. Lindenstrauss and M. Tsukamoto. Mean dimension of $\mathbb{Z}^k$-actions. \textit{Geom. Func. Anal.} \textbf{26} (2016), 778-817.

\bibitem{G-K}J. Giol and D. Kerr. Subshifts and perforation.\textit{ J. Reine. Angew. Math.} \textbf{639} (2010), 107-119.

\bibitem{G-L-N} G. Gong, H. Lin, and Z. Niu. Classification of finite simple amenable $\mathcal{Z}$-stable $C^\ast$-algebras. arXiv:1501.00135

\bibitem{G-W-Y}E. Guentner, R. Willett and G. Yu. Dynamic asymptotic dimension: relation to dynamics, topology, coarse geometry, and $C^\ast$-algebras. \textit{Math. Annalen} \textbf{367} (2017), 785-829.

\bibitem{H-W-Z}I. Hirshberg, W. Winter and J. Zacharias. Rokhlin dimension and C*-dynamics. \textit{Comm. Math. Phys.}, \textbf{335} (2015), 637-670.

\bibitem{Kech}A. S. Kechris.\textit{ Classical Descriptive set theory}. Graduate Texts in Mathematics. 156, Springer-Verlag, New York, 1995.

\bibitem{D}D. Kerr. Dimension, comparison, and almost finiteness. arXiv:1710.00393

\bibitem{Kerr-L} D. Kerr and H. Li. \textit{Ergodic theory: Independence and Dichotomies}. Springer, 2016.


\bibitem{K-R}E. Kirchberg and M. R{\o}rdam. Central sequence $C^\ast$-algebras and tensorial absorption of the Jiang-Su algebra. \textit{J. Reine Angew. Math.} \textbf{695} (2014), 175-214.

\bibitem{Lin}H. Lin. Simple nuclear $C^\ast$-algebras of tracial topological rank one. \textit{J. Funct. Anal.} \textbf{251} no.2 (2007), 601-679.

\bibitem{L-P}H. Lin and N. C. Phillips. Crossed products by minimal homeomorphisms. \textit{J. Reine. Angew. Math.} \textbf{641} (2010), 95-122.

\bibitem{Linden}E. Lindenstrauss. Mean dimension, small entropy factors, and an embedding theorem. \textit{Publ. Math. IHES} \textbf{89} (1999), 227-262.

\bibitem{L-W}E. Lindenstrauss and B. Weiss. Mean topological dimension. \textit{Israel J. Math.} \textbf{115} (2000), 1-24.


\bibitem{Matui}H. Matui. Homology and topological full groups of \'{e}tale groupoids on totally disconnected spaces. \textit{Proc. Lond. Math. Soc.} (3) \textbf{104} (2012), 27-56.

\bibitem{M-Sato}H. Matui and Y. Sato. Decomposition rank of UHF-absorbing $C^\ast$-algebras. \textit{Duke Math. J.} \textbf{163} (2014), No. 14, 2687--2708.


\bibitem{M-S}H. Matui and Y. Sato. Strict comparison and $\mathcal{Z}$-absorption of nuclear $C^\ast$-algebras. \textit{Acta Math.} \textbf{209} (2012), 179-196.


\bibitem{NP}N. C. Phillips. \textit{An introduction to crossed product $C^\ast$-algebras and minimal dynamics.} http://pages.uoregon.edu/ncp/Courses/CRMCrPrdMinDyn/$\textrm{Notes\_20170205}$.pdf, 2017.

\bibitem{Phillips}N. C. Phillips. Crossed products of the Cantor set by free minimal actions of $\mathbb{Z}^d$. \textit{Comm. Math. Phys.} \textbf{256} (2005), 1-42.

\bibitem{R}M. R{\o}rdam. The stable and the real rank of $\mathcal{Z}$-absorbing $C^\ast$-algebras. \textit{Internat. J. Math.} \textbf{15} (2004), 1065-1084.

\bibitem{Ro}M. R{\o}rdam. Structure and classification of $C^\ast$-algebras. \textit{Proceedings of the International Congress of Mathematicians}(Madrid 2006), Volume II, EMS Publishing House, Zurich 2006, 1581-1598.

\bibitem{S}Y. Sato. Trace spaces of simple nuclear $C^\ast$-algebras with finite-dimensional extreme boundary. arXiv:1209.3000.

\bibitem{S-W-W}Y. Sato, S. White and W. Winter. Nuclear dimension and $\mathcal{Z}$-stability. \textit{Invent. Math.} \textbf{202} (2015), 893-921.

\bibitem{Shub-Weiss} M. Shub and B. Weiss. Can one always lower topological entropy?. \textit{Ergodic Theory and Dynamical Systems} \textbf{11} (1991), 535-546.


\bibitem{Suzuki}Y. Suzuki. Almost finiteness for general \'{e}tale groupoids and its applications to stable rank of crossed products. arXiv:1702.04875v1.

\bibitem{Sz}G. Szab\'{o}. The Rokhlin dimension of topological $\mathbb{Z}^m$-actions. \textit{Proc. Lond. Math. Soc.} \textbf{110} (2015), no. 3, 673-694.

\bibitem{S-W-Z}G. Szab\'{o}, J. Wu and J. Zacharias. Rokhlin dimension for actions of residually finite groups. To appear in \textit{Ergodic Theory Dynam. Systems.}

\bibitem{Ti-W-W}A. Tikuisis, S. White, and W. Winter. Quasidiagonality of nuclear $C^\ast$-algebras. \textit{Ann. of
Math.} (2) \textbf{185} (2017), 229-284.

\bibitem{T-W-W}A. S. Toms, S. White and W. Winter. $\mathcal{Z}$-stability and finite-dimensional tracial boundaries. \textit{Int. Math. Res. Not.}, IMRN \textbf{2015}, No.10, 2702-2727.

\bibitem{Tu}J.-L. Tu. La conjecture de Baum-Connes pour les feuilletages moyennables. \textit{K-theory} \textbf{17} (1999), 215-264.

\bibitem{Winter}W. Winter. Decomposition rank and $\mathcal{Z}$-stability \textit{Invent. Math.} \textbf{179} (2010), 229-301.

\bibitem{W}W. Winter. Nuclear dimension and $\mathcal{Z}$-stability of pure $C^\ast$-algebras. \textit{Invent. Math.} \textbf{187} (2012), 259-342.

\bibitem{Win-ICM}W. Winter. Structure of nuclear C*-algebras: From quasidiagonality to classification, and back again. arXiv: 1712.00247.

\bibitem{W-Z}W. Winter and J. Zacharias. The nuclear dimension of $C^\ast$-algebras.\textit{ Adv. Math.} \textbf{224} (2010), 461-498.
\end{thebibliography}
\end{document}